\pdfoutput=1
\documentclass[12pt]{amsart}

\usepackage{amsmath,amssymb,amsxtra,times,anysize,youngtab}
\usepackage{graphicx,tabu,slashbox,pict2e,xcolor,thmtools,microtype}
\usepackage{hyperref,cleveref}
\marginsize{1in}{1in}{1in}{1in}

\definecolor{darkblue}{RGB}{0,0,96}
\hypersetup{colorlinks={true},linkcolor={darkblue},citecolor={darkblue},filecolor={darkblue},urlcolor={darkblue},pdfauthor={Eugene Gorsky and Lukas Lewark},pdftitle={On stable sl3-homology of torus knots}}

\declaretheorem[name=Theorem]{theorem}
\declaretheorem[name=Conjecture,sibling=theorem]{conjecture}
\crefname{conjecture}{conjecture}{conjectures}
\declaretheorem[name=Remark,sibling=theorem]{remark}
\declaretheorem[name=Example,sibling=theorem]{example}

\newcommand{\CH}{\mathcal{H}}
\newcommand{\CP}{\mathcal{P}}
\newcommand{\BZ}{\mathbb{Z}}
\newcommand{\BQ}{\mathbb{Q}}
\newcommand{\hy}{-}

\DeclareMathOperator{\SL}{\mathit{sl}}
\DeclareMathOperator{\ALG}{alg}
\DeclareMathOperator{\RED}{red}

\newcommand{\Thmref}[1]{\hyperref[#1]{\nameCref{#1}~\ref{#1}}}
\newcommand{\thmref}[1]{\Thmref{#1}}

\title{On stable $\SL_3$\hy homology of torus knots}
\author{Eugene Gorsky}
\address{Department of Mathematics, Columbia University \newline 2990 Broadway
New York, NY 10027 USA}
\email{egorsky@math.columbia.edu}
\author{Lukas Lewark}
\thanks{The first author is partially supported by the grants RFBR–13-01-00755, NSh-4850.2012.1.
The second author is supported by the EPSRC-grant EP/K00591X/1.}
\address{Department of Mathematical Sciences,
Durham University\newline 
Science Labs,
South Road, Durham
DH1 3LE, UK }
\email{lukas.lewark@durham.ac.uk}
\date{April 2, 2014}
\begin{document}
\maketitle
\begin{abstract}
The stable Khovanov-Rozansky homology of torus knots has been conjecturally described as the Koszul homology of an explicit non-regular sequence of polynomials.
We verify this conjecture against newly available computational data for $\SL_3$-homology.
Special attention is paid to torsion. In addition, explicit conjectural formulae are given for the $\SL_N$-homology of $(3,m)$-torus knots for all $N$ and $m$, which are motivated by higher categorified Jones-Wenzl projectors.
Structurally similar formulae are proven for Heegard-Floer homology.
\end{abstract}
%\tableofcontents

\section{Introduction}

Torus knots are among the simplest and best understood knots, and their invariants have been studied and computed using various mathematical and physical tools. For example, Jones and Rosso \cite{RJ} computed all $\SL_N$\hy quantum invariants of torus knots explicitly. 

The computation of Khovanov and Khovanov-Rozansky homology of torus knots remains an important open problem in knot theory. The experimental computations of Bar-Natan, Shumakovitch, Turner \cite{BN,katlas,shu,turner} and others revealed a rich structure, including nontrivial torsion, in $\SL_2$\hy homology of torus knots.
On the other hand, the main conjecture of \cite{GORS} related the HOMFLY-PT homology of $(n,m)$\hy torus knots to finite dimensional representations of rational Cherednik algebras with parameter $m/n$. In the limit $m\to \infty$, this construction has been used in \cite{GOR} for a conjectural description of stable $\SL_N$\hy homology  of torus knots. For $N=2$, this conjecture has been verified against the experimental data for stable torus knots with up to 8 strands, and they agree both in reduced and in unreduced homology. 

\begin{restatable}{conjecture}{conjOne}(\cite{GOR})
\label{main conj}
Consider the triply graded algebra $\CH_n$ with even generators $x_0,\ldots,x_{n-1}$ and odd generators $\xi_0,\ldots,\xi_{n-1}$
of degrees
$$
\deg x_k=q^{2k+2}t^{2k},\quad \deg \xi_{k}=a^2q^{2k}t^{2k+1}.
$$
It is equipped with the family of Koszul differentials $d_N$ defined by the equations:
\begin{equation}
\label{def d}
d_N(x_k)=0,\quad d_N(\xi_k)=\sum_{i_1+\ldots+i_N=k}x_{i_1}\cdot\ldots \cdot x_{i_N}.
\end{equation}
Then the stable $\SL_N$\hy Khovanov-Rozansky homology of the $(n,\infty)$\hy torus knot is isomorphic to $H^{*}(\CH_n,d_N)$
after the regrading $a=q^N$.
\end{restatable}

Recently the second named author developed a computer program \cite{L,foam} that allows one to compute integral $\SL_3$\hy Khovanov-Rozansky homology of any knot, in reasonable time for knots such as the $(5,14)$\hy\ or $(6,7)$\hy torus knots.
The main goal of this article is to verify \thmref{main conj} for $N=3$ using the newly available experimental data, see \thmref{sec4}.

In \thmref{sec2}, we study \thmref{main conj}, proving it to be correct for the $(2,\infty)$\hy torus knot,
and giving an explicit formula for the 3\hy strand case. Reduced homology is related to unreduced. 
We show that for a prime $p > N + 1$, the Koszul model for stable $\SL_N$\hy homology of the $p$\hy strand torus knot has $p$\hy torsion, and give a formula for the Koszul model of $\SL_N$\hy homology with $\BZ_N$\hy coefficients, which has a relatively simple form.
Note that while a definition of $\SL_N$\hy homology with integer or $\BZ_p$\hy coefficients may be expected to exist, it has so far only been given for $N = 2$ and $N = 3$.

For 3\hy strand finite torus knots $T(3,m)$, we also write down explicit conjectural formulae for the Poincar\'e polynomials of $\SL_N$\hy homology (see \thmref{sec3}).  Namely, for each of the 4 standard Young tableaux $T$ with 3 boxes we define a certain explicit rational function $\CP_{T}(N)$.

\begin{restatable}{conjecture}{conjTwo}
\label{conj:projectors}
The Poincar\'e series of the unreduced $\SL_N$\hy homology of the $(3,3k+1)$\hy torus knot equals (up to an overall shift of $q^{3k(N-1)-2}$):
$$
\CP(T(3,3k+1),\SL_N)=\CP_{\text{\scriptsize\young(123)}}(N)+q^{6k}t^{4k}\CP_{\text{\scriptsize\young(12,3)}}(N)+q^{6k}t^{4k}\CP_{\text{\scriptsize\young(13,2)}}(N)+q^{12k}t^{6k}\CP_{\text{\scriptsize\young(1,2,3)}}(N).
$$
The Poincar\'e series of the unreduced $\SL_N$\hy homology of the $(3,3k+2)$\hy torus knot equals (up to an overall shift of $q^{3k(N-1)-3}$):
\begin{align*}
\CP(T(3,3k+2),\SL_N)=\ &\CP_{\text{\scriptsize\young(123)}}(N)+q^{6k}t^{4k}\CP_{\text{\scriptsize\young(12,3)}}(N)+ \\[2ex]
&q^{6k+4}t^{4k+2}\CP_{\text{\scriptsize\young(13,2)}}(N)+q^{12k+4}t^{6k+2}\CP_{\text{\scriptsize\young(1,2,3)}}(N).
\end{align*}
\end{restatable}

The conjecture is motivated by the work \cite{GN}, where the ``refined Chern-Simons invariants''  \cite{AS}  of $(n,m)$\hy torus knots were expressed as sums 
over standard Young tableaux of size $n$. It was conjectured in \cite{AS} that these invariants agree with the Poincar\'e polynomials of HOMFLY-PT homology.
We show that for $n=3$ (and $n=2$) the HOMFLY-PT analogues of  $\CP_{T}(N)$ can be viewed as Poincar\'e series of certain free supercommutative algebras, and $\CP_{T}(N)$ can be obtained as homology of certain differentials $d_N$ acting on these algebras. In particular, $\CP_{\text{\scriptsize\young(123)}}(N)$
coincides with the Poincar\'e series of stable $\SL_N$\hy homology of $T(3,\infty)$ which, by \thmref{main conj}, can be described as homology of $d_N$.
This decomposition is also expected to be related to the categorification of higher Jones-Wenzl projectors \cite{CH}, but we do not pursue this relation here.

Surprisingly enough, we prove that the Poincar\'e polynomials of the Heegaard-Floer homology of $T(3,m)$ can be written using a formula with a similar structure. The corresponding differential corresponds to $d_0$ in the notations of \cite{DGR,GORS} and to $d_{1|1}$ in the notations of \cite{GGS}.

\section*{Acknowledgments}

The authors would like to thank M. Khovanov, A. Lobb and J. Rasmussen for the useful discussions and remarks.

\section{\texorpdfstring{The Koszul model for stable $\SL_N$\hy homology}{The Koszul model for stable slN homology}}

\label{sec2}

\subsection{Koszul model}

In \cite{GOR,GORS} the following algebraic model for the stable  homology of torus knots was proposed. 

\conjOne*

\begin{example}
For $n=1$ one has a single even generator $x_0$ and a single odd generator $\xi_0$, and the differential has the form $d_N(\xi_0)=x_0^N$.
The homology has dimension $N$ and is spanned by $1, x_0, \ldots, x_0^{N-1}$, so it is indeed isomorphic to the $\SL_N$\hy homology of the unknot.
\end{example}

We denote the homology of $d_N$ as $H^{\ALG}_{n}(\SL_N)$. After setting $a=q^N$ the variables are graded as follows:
$$
\deg x_k=q^{2k+2}t^{2k},\quad \deg \xi_{k}=q^{2N+2k}t^{2k+1},
$$
so that $d_N$ preserves the $q$-grading and decreases the $t$-grading by 1.

\subsection{Generators}

Consider the generating functions $x(\tau)=\sum_{k=0}^{\infty} x_k \tau^{k}$ and $\xi(\tau)=\sum_{k=0}^{\infty}\xi_k \tau^{k}.$
The equation \eqref{def d} can be rewritten as 
\begin{equation}
\label{d xi}
d_N(x(\tau))=0,\quad d_N(\xi(\tau))=x(\tau)^N \mod \tau^n.
\end{equation}
Remark that
\begin{equation}
\label{d xi dot}
d_N(\dot{\xi}(\tau))=N\dot{x}(\tau)x(\tau)^{N-1} \mod \tau^n.
\end{equation}
Define
$$
\mu(\tau)=\sum_{k=1}^{\infty}\mu_{k} \tau^{k-1}:=N\dot{x}(\tau)\xi(\tau)-x(\tau)\dot{\xi}(\tau),
$$
so $\mu_k=\sum_{i+j=k}(Ni-j)x_i\xi_j.$
It follows from \eqref{d xi} and \eqref{d xi dot} that $d_N(\mu(\tau))=0$, so 
$d_N(\mu_k)=0$ for all $k$. The degree of $\mu_k$ equals $q^{2N+2k+2}t^{2k+1}$.

\begin{conjecture}
\label{conj:generators}
The homology of $d_N$ is generated as an algebra by $x_k$ and $\mu_k$.
\end{conjecture}

\subsection{Reduced homology}

The reduced homology is defined as a quotient of the unreduced homology by the homology of the unknot, i.e.
by the equation $x_0=0$.
It turns out that the reduced and unreduced Koszul homology are tightly related.

\begin{theorem}
We have the following isomorphism:
$$
H^{\ALG,\RED}_{n}(\SL_N)\simeq \BZ[\xi_1,\ldots,\xi_{N-1},x_{n-N+1},\ldots,x_n]\otimes H^{\ALG}_{n-N}(\SL_N).
$$
\end{theorem}

\begin{proof}
In reduced homology we set $x_0=0$, so $d_N(\xi_i)=0$ for $i<N$.
Furthermore, $d_N^{\RED}(\xi_i)$ coincides with $d_N(\xi_{i-N})$ where all 
$x_i$ are replaced by $x_{i-1}$. For example, $d_{N}^{\RED}(\xi_N)=x_1^{N}$.
\end{proof}

 Note that the isomorphism does not preserve $q$\hy\ and $t$\hy gradings,
 but it is easy to reconstruct the change of gradings transforming $x_i$ to $x_{i+1}$
 and $\xi_i$ to $\xi_{i+N}$. 

\subsection{$\BZ_N$\hy homology for prime $N$}

In this subsection we give a simple description of $\SL_N$\hy homology with $\BZ_N$\hy coefficients, 
assuming for simplicity that $N$ is a prime number.

\begin{theorem}\label{thm:ZNZcoefficients}
Suppose that $N$ is prime. Then the algebraic Poincar\'e series has the form:
$$\CP^{\ALG}_n(\SL_N,\BZ_N)=\prod_{k=0}^{n-1}\frac{1+q^{2N+2k}t^{2k+1}}{1-q^{2k+2}t^{2k}}
\prod_{i=0}^{\lfloor\frac{n-1}{N}\rfloor}\frac{1-q^{2N+2iN}t^{2iN}}{1+q^{2N+2iN}t^{2iN+1}}.$$
\end{theorem}
\begin{proof}
If $N$ is prime, then
$$d_N(\xi(\tau))=x(\tau)^N\equiv\sum x_i^{N}\tau^{iN}\mod N.$$
Therefore 
$$d_N(\xi_k)\equiv \begin{cases}
0\mod N,&  k\neq Ni\\
x_i^{N}\mod N,& k=Ni.\\
\end{cases}$$
The homology is generated by $x_i$ modulo relations $x_i^N=0$ and by $\xi_k$ for $k$ not divisible by $N$.
\end{proof}
See \thmref{table2} for an illustration of the theorem.

\subsection{Torsion}
\begin{table}[p]
\rotatebox{90}{\parbox{\textheight}{\centering
    \newcommand{\RH}[1]{\makebox[0.86em]{#1}}
    \newcommand{\FT}{\raisebox{-0.8ex}{\makebox[0pt]{\hspace{-1em}\textbf{\textcolor{green}{1}}}}}
\newcommand{\FTT}{\raisebox{-0.8ex}{\makebox[0pt]{\hspace{-1em}\textbf{\textcolor{green}{2}}}}}
\newcommand{\HH}[1]{\makebox[0pt]{#1}}
\extrarowsep=1.15mm
\begin{tabu}{|[lightgray]c|*{15}{c|[lightgray]}c||*{15}{c|[lightgray]}}
\tabucline[lightgray]{1-31}
\backslashbox{$\ddagger$\kern-0.1em}{\kern-0.1em$t$} & \RH{0} & \RH{1} & \RH{2} & \RH{3} & \RH{4} & \RH{5} & \RH{6} & \RH{7} & \RH{8} & \RH{9}  & \RH{10} & \RH{11} & \RH{12} & \RH{13} & \RH{14} & \RH{15} & \RH{16} & \RH{17} & \RH{18} & \RH{19} & \RH{20} & \RH{21} & \RH{22} & \RH{23} & \RH{24}  & \RH{25} & \RH{26} & \RH{27} & \RH{28} & \RH{29} \\\hline
$6  $  &   &   &   & 1 &   & 1 &   & 1 & 1 &   & 1 &   & 1 &    &   &    &   &    &   &    &    &    &   &   &   &   &   &   &   & \\\tabucline[lightgray]{1-31}
$4  $  & 1 &   & 1 & 1 & 1 & 2 &   & 3 & 1 & 2 & 1 & 1 & 2 &    & 1 &    & 1 &    &   &    &    &    &   &   &   &   &   &   &   & \\\tabucline[lightgray]{1-31}
$2  $  & 1 &   & 1 & 1 & 2 & 2 & 1 & 3 & 2 & 4 & 1 & 4 & 3 & 3  & 3 & 1  & 4 & 1  & 2 & 1  & 1  & 1  &   &   &   &   &   &   &   & \\\tabucline[lightgray]{1-31}
$0  $  & 1 &   & 1 &   & 2 & 1 & 2 & 2 & 3 & 5 & 2 & 7 & 3 & 8  & 4 & 6  & 7 & 4  & 6 & 2  & 5  & 2  & 2 & 1 & 1 & 1 &   &   &   & \\\tabucline[lightgray]{1-31}
$-2 $  &   &   &   &   & 1 &   & 2 &   & 4 & 3 & 4 & 6 & 5 & 10 & 5 & 11 & 7 & 11 & 8 & 8  & 9  & 5  & 7 & 3 & 5 & 2 & 2 & 1 & 1 & \\\tabucline[lightgray]{1-31}
$-4 $  &   &   &   &   &   &   & 1 &   & 3 & 1 & 4 & 3 & 6 & 7  & 6 & 11 & 8 & 14 & 8 & 13 & 10 & 11 & 9 & 8 & 8 & 5 & 4 & 2 & 2 & 1\\\tabucline[lightgray]{1-31}
$-6 $  &   &   &   &   &   &   &   &   & 1 &   & 2 &   & 5 & 2  & 6 & 6  & 7 & 11 & 7 & 13 & 9  & 13 & 9 & 11& 8 & 7 & 5 & 4 & 3 & 1\\\tabucline[lightgray]{1-31}
$-8 $  &   &   &   &   &   &   &   &   &   &   &   &   & 2 &    & 4 & 1  & 6 & 4  & 6 & 7  & 7  & 10 & 6 & 9 & 6 & 7 & 4 & 3 & 2 & 1\\\tabucline[lightgray]{1-31}
$-10$  &   &   &   &   &   &   &   &   &   &   &   &   &   &    & 1 &    & 2 &    & 4 & 1  & 5  & 3  & 4 & 4 & 3 & 4 & 2 & 2 & 1 & \\\tabucline[lightgray]{1-31}
$-12$  &   &   &   &   &   &   &   &   &   &   &   &   &   &    &   &    &   &    & 1 &    & 2  &    & 2 &   & 2 & 1 & 1 &   &   & \\\tabucline[lightgray]{1-31}
\end{tabu}\\[\baselineskip]
\caption{The unreduced $\SL_3$\hy homology of $T(5,9)$ with $\BZ_3$\hy coefficients, cf. \thmref{thm:ZNZcoefficients}.
The top row indicates the homological degree $t$, and the left-most column $\ddagger = q - 2t$, where $q$ is the quantum degree.
For simplicity, we follow the grading convention (which is different from the Khovanov and Rozansky's original one)
that leads to e.g. the reduced $\SL_3$\hy homology of the trefoil with positive crossings
having Poincar\'e polynomial $q^4 + t^2q^8 + t^3q^{12}$. 
Note also that the stable part of an actual $(n,m)$-torus knot is shifted in the quantum degree by $2(nm - n - m)$ relative to the homology of the $(n,\infty)$-torus knot.
Each cell shows the dimension of homology at the corresponding degree.
The first divergence from the stable model (at homological degree 16) is indicated by a doubled vertical line.
}\label{table2}}}
\end{table}
\begin{table}[p]
\rotatebox{90}{\parbox{\textheight}{\centering
\newcommand{\RH}[1]{\makebox[0.86em]{#1}}
\newcommand{\FT}{\raisebox{-0.8ex}{\makebox[0pt]{\hspace{-1em}\textbf{\textcolor{green}{1}}}}}
\newcommand{\FBT}{\raisebox{-0.625ex}{\makebox[0pt]{\fbox{\rule{0pt}{1.025em}\hspace{1em}}}}\raisebox{-0.8ex}{\makebox[0pt]{\hspace{-1em}\textbf{\textcolor{green}{1}}}}}
\newcommand{\FTT}{\raisebox{-0.8ex}{\makebox[0pt]{\hspace{-1em}\textbf{\textcolor{green}{2}}}}}
\newcommand{\HH}[1]{\makebox[0pt]{#1}}
\extrarowsep=1.15mm
\begin{tabu}{|[lightgray]c|*{15}{c|[lightgray]}c||*{14}{c|[lightgray]}}
\tabucline[lightgray]{1-30}
\backslashbox{$\ddagger$\kern-0.1em}{\kern-0.1em$t$} & \RH{0} & \RH{1} & \RH{2} & \RH{3} & \RH{4} & \RH{5} & \RH{6} & \RH{7} & \RH{8} & \RH{9}  & \RH{10} & \RH{11} & \RH{12} & \RH{13} & \RH{14} & \RH{15} & \RH{16} & \RH{17} & \RH{18} & \RH{19} & \RH{20} & \RH{21} & \RH{22} & \RH{23} & \RH{24}  & \RH{25} & \RH{26} & \RH{27} & \RH{28} \\\hline
$6$ & & & & \HH{1} & & \HH{1} & & & \HH{1} & & \HH{1} & & \HH{1} & & & & & & & & & & & & & & & & \\\tabucline[lightgray]{1-30}
$4$ & \HH{1} & & & \HH{1} & & \HH{2} & & \HH{2} & & \HH{1} & \HH{1} & & \HH{2} & & \HH{1} & & & & & & & & & & & & & & \\\tabucline[lightgray]{1-30}
$2$ & \HH{1} & & \HH{1} & & \HH{1} & \HH{1} & & \HH{3} & & \HH{4} & & \HH{2} & \HH{2} & & \HH{3} & & \HH{3} & & \HH{1} & \HH{1} & & \HH{1} & & & & & & & \\\tabucline[lightgray]{1-30}
$0$ & \HH{1} & & \HH{1} & & \HH{2} & & \HH{2} & \HH{1} & \HH{1} & \HH{4} & & \HH{6} & & \HH{5} &\FT\HH{1} & \HH{1} & \HH{5} & & \HH{5} & & \HH{3} & \HH{1} & & \HH{1} & & \HH{1} & & & \\\tabucline[lightgray]{1-30}
$-2$ & & & & & \HH{1} & & \HH{2} & & \HH{3} & \HH{1} & \HH{1} & \HH{4} & & \HH{8} & & \HH{8} &\FT\HH{1} & \HH{4} &\FT\HH{5} & & \HH{7} & & \HH{5} & \FT & \HH{3} & \HH{1} & \HH{1} & \HH{1} & \\\tabucline[lightgray]{1-30}
$-4$ & & & & & & & \HH{1} & & \HH{3} & & \HH{4} & \FBT & \HH{4} &\FT\HH{3} & \HH{1} & \HH{9} & & \HH{11}& \FT & \HH{8} &\FTT\HH{2}& \HH{3} &\FT\HH{5} & & \HH{6} & & \HH{3} & & \HH{1}\\\tabucline[lightgray]{1-30}
$-6$ & & & & & & & & & \HH{1} & & \HH{2} & & \HH{5} & & \HH{5} &\FT\HH{2} & \HH{3} &\FT\HH{6} & & \HH{10}& & \HH{10}& \FTT& \HH{7} &\FTT\HH{1}& \HH{3} &\FT\HH{2} & \HH{1} & \HH{1}\\\tabucline[lightgray]{1-30}
$-8$ & & & & & & & & & & & & & \HH{2} & & \HH{4} & & \HH{6} &\FT & \HH{5} &\FTT\HH{2}& \HH{3} &\FT\HH{6} & & \HH{7} & & \HH{5} & & \HH{2} & \\\tabucline[lightgray]{1-30}
$-10$ & & & & & & & & & & & & & & & \HH{1} & & \HH{2} & & \HH{4} & & \HH{5} & \FT & \HH{4} & \FTT& \HH{2} &\FT\HH{1} & \HH{1} & \HH{1} & \\\tabucline[lightgray]{1-30}
$-12$ & & & & & & & & & & & & & & & & & & & \HH{1} & & \HH{2} & & \HH{2} & & \HH{2} & & \HH{1} & & \\\tabucline[lightgray]{1-30}
\end{tabu}\\[\baselineskip]
\caption{The unreduced $\SL_3$-homology of $T(5,9)$. See the caption of \thmref{table2} for detailed explanations. 
The $\BZ_5$\hy torsion is printed at the lower left corner of each cell in bold and green; for the sake of legibility, $\BZ_3$\hy torsion is not printed, but cf. \thmref{table2}.
The only other torsion is $\BZ_2$\hy torsion $(t^{23}q^{36} + t^{24}q^{42})$, which is not printed either, and may be expected to be unstable.
Note the presence of $\BZ_5$\hy torsion at $t^{11} q^{18}$ (indicated by a box), in accordance with \thmref{thm:torsion}.}
\label{table1}}}
\end{table}
The computations of Khovanov homology showed the presence of large and complicated torsion. The first computations in $\SL_3$\hy knot homology indeed confirmed the existence of torsion as well. It has been shown in \cite{GOR} that the algebraic model is compatible at least with some part of this torsion, the following result generalizes this check to $\SL_N$\hy homology.

\begin{theorem}\label{thm:torsion}
Let $p>N+1$ be a prime number, then the algebraic homology $H^{\ALG}_{p}(\SL_N,\BZ)$ has $p$-torsion in bidegree $q^{2p+2N+2}t^{2p+1}$.
\end{theorem}

\begin{proof}
Recall that for $\mu_p=\sum_{i=0}^{p}(Ni-p+i)x_i\xi_{p-i}$ one has $d_N(\mu_p)=0$.
This generator is not present in $\CH_p$, since $x_p$ and $\xi_p$ are not present in the Koszul complex.
Consider the element 
$$t_p:=\sum_{i=1}^{p-1}(Ni-p+i)x_i\xi_{p-i}=\mu_p-Npx_p\xi_0+px_0\xi_p,$$
which is present in $\CH_p$.
We have
$$d_N(t_p)=p(x_0d_N(\xi_p)-Nx_px_0^N),$$
so $d_N(t_p)$ vanishes modulo $p$. Since $t_p$ has a term $(N-p+1)x_1\xi_{p-1}$ that does not vanish modulo $p$,
the dimension of the kernel of $d_N$ over $\BZ_p$ is bigger than its rank over $\BZ$.
\end{proof}

\begin{example}
For $N=3$ and $p=5$ the computations of \cite{foam} show $\BZ_5$\hy torsion in 
bidegree $q^{18}t^{11}$ of the homology of the $(5,m)$\hy torus knot for $6 \leq m \leq 14$.
It is plausible that it survives in the stable limit and hence agrees with the torsion generator $t_5$.
See \thmref{table1} for the case $m = 9$.
\end{example}

Indeed, the torsion classes described in \thmref{thm:torsion} (and their products with $x$'s and $\mu$'s) do not cover all stable torsion.

\begin{example}
Consider the following class in $\CH_n$ (defined for $n\ge 5$):
$$
A:=4x_0\xi_1\xi_4-8x_1\xi_0\xi_4+8x_4\xi_0\xi_1-2x_0\xi_2\xi_3-2x_2\xi_1\xi_2-4x_3\xi_0\xi_2+x_1\xi_1\xi_3+4x_2\xi_0\xi_3.
$$
One can check that $d_2(A)=10x_1x_3(2x_1\xi_0-x_0\xi_1)$, so $A$ represents a class in $H^{\ALG}_{n}(\SL_2,\BZ_5)$.
On the other hand, one can check (say, from \thmref{conj:generators}) that $A$ cannot be lifted to a class in $H^{\ALG}_{n}(\SL_2,\BZ)$,
hence $H^{\ALG}_{n}(\SL_2,\BZ)$ has 5-torsion in bidegree $q^{20}t^{12}$ for all $n\ge 5$. This torsion can be seen in the  data of \cite{katlas},
      e.g. for the $(5,6)$\hy, $(6,7)$\hy, $(7,8)$\hy\ and $(8,9)$\hy\  torus knots.
\end{example}

\begin{example}
\label{ex:7 torsion}
Consider the following class in $\CH_6$: 
$$
B:=5x_1^2\xi_5-5x_1x_5\xi_1-10x_1x_4\xi_2+16x_2x_4\xi_1-15x_1x_3\xi_3+15x_3^2\xi_1+3x_2^2\xi_3-3x_2x_3\xi_2-6x_1x_2\xi_4.
$$
One can check that $d_3(B)=-105x_0x_1^2x_2x_3$, so $B$ represents a class in $H^{\ALG}_{6}(\SL_3,\BZ_7)$.
Moreover, $B$ cannot be lifted to a class in $H^{\ALG}_{6}(\SL_3,\BZ)$, so $H^{\ALG}_{6}(\SL_3,\BZ)$ has 7-torsion in bidegree $q^{24}t^{15}$.
This is also confirmed by the experimental data \cite{foam}, i.e. by the computed homology of the $(6,7)$\hy torus knot and $(6,8)$\hy torus link.
Note that $B\equiv x_2\mu_5\mod 5$, so it does not contribute to the $5$-torsion. 
\end{example}

\begin{remark}
Przytycki and Sazdanovi\'c conjectured in \cite[Conjecture 6.1]{PS} that the closure of an $n$-strand braid cannot have $\BZ_p$-torsion in Khovanov homology
for a prime $p > n$. Example \ref{ex:7 torsion} shows that a naive generalization of this conjecture
(``the closure of an $n$-strand braid cannot have $\BZ_p$-torsion in $\SL_N$\hy homology for a prime $p > \max(n,N)$'')
does not hold for $\SL_3$\hy homology, since the $(6,7)$\hy torus knot and the $(6,8)$\hy torus link already have 7-torsion.
\end{remark}

\subsection{$(2,\infty)$ case}

For $n=2$ we have two even generators $x_0, x_1$ and two odd generators $\xi_0,\xi_1$, with the differential
$$
d_N(\xi_0)=x_0^N,\quad d_N(\xi_1)=Nx_0^{N-1}x_1.
$$
There is a generator $\mu_1=Nx_1\xi_0-x_0\xi_1$ of degree $q^{2N+4}t^{3}$, the homology is generated by $x_0,x_1$ and $\mu_1$ modulo relations
$$
x_0^{N}=0,\quad Nx_0^{N-1}x_1=0,\quad x_0^{N-1}\mu_{1}=0.
$$
The monomial basis in $\mathbb{Q}$-homology has the form
$$x_0^{N-1},\ x_0^{i}x_1^{a}\mu_1^{b}, \ i< N-1, b\le 1,$$
so the Poincar\'e series has the form
\begin{align}
\label{n=2}
\CP_2(q,t)& =  q^{2N-2}+\frac{(1-q^{2N-2})(1+q^{2N+4}t^3)}{(1-q^2)(1-q^4t^2)}  \\[1ex]
          &  = \frac{1-q^{2N}-q^{2N+2}t^2+q^{2N+4}t^2+q^{2N+4}t^3-q^{4N+2}t^3}{(1-q^2)(1-q^4t^2)}.\notag
\end{align}

The actual $\SL_N$\hy Khovanov-Rozansky homology was computed for all $N$ by Cautis in \cite[Corollary 10.2]{C}.
Since the two answers coincide, we obtain the following result.

\begin{theorem}
\Thmref{main conj} is true for $n=2$ and all $N$.
\end{theorem} 

\subsection{$(3,\infty)$ case}

For $n=3$ 	we have new generators $x_2$ and $\xi_2$ in $\CH_3$, and the differential looks like
$$d_N(\xi_2)=Nx_0^{N-1}x_2+\binom{N}{2}x_0^{N-2}x_1^{2}.$$
We have a new homology generator
$$\mu_2=2Nx_2\xi_0+(N-1)x_1\xi_1-2x_0\xi_2,\qquad \deg \mu_2=q^{2N+6}t^3$$
and one can explicitly check the identity $d_N(\mu_2)=0$.

In zeroth Koszul homology we get $$\BQ[x_0,x_1,x_2]/\left(x_0^N,Nx_0^{N-1}x_1,Nx_0^{N-1}x_2+\binom{N}{2}x_0^{N-2}x_1^{2}\right).$$
If we multiply the third term by $x_1$ and subtract $x_2$ times the second one, we get $x_0^{N-2}x_1^3=0$. 
The monomial basis consists of all monomials not divisible by $x_0^{N}, x_0^{N-1}x_1, x_0^{N-1}x_2$ and  $x_0^{N-2}x_1^3$,
and the Poincar\'e series equals
\begin{align*}
\CP_3^{a=0}(q,t) & =q^{2N-2}+\frac{1-q^{2N-2}}{(1-q^2)(1-q^4t^2)(1-q^6t^4)}-\frac{q^{2N+8}t^6}{(1-q^4t^2)(1-q^8t^6)} \\[1ex]
         & = \frac{1-q^{2N}-q^{2N+2}t^2-q^{2N+4}t^4+q^{2N+4}t^2+q^{2N+6}t^4}{(1-q^2)(1-q^4t^2)(1-q^6t^4)}.
\end{align*}
Let us describe the relations between $\mu$'s.
We have 
\begin{align*}
d_N(\xi_0\xi_1) & =x_0^{N-1}\mu_1, \\
 d_N(2\xi_0\xi_2) & =2x_0^{N}\xi_2-\left(2Nx_0^{N-1}x_2+N(N-1)x_0^{N-2}x_1^{2}\right)\xi_0 \\
  & =-x_0^{N-1}\mu_2-(N-1)x_0^{N-2}x_1\mu_1, \\
d_N(2\xi_1\xi_2) & =2Nx_0^{N-1}x_1\xi_2-\left(2Nx_0^{N-1}x_2+N(N-1)x_0^{N-2}x_1^{2}\right)\xi_1 \\
    & = -Nx_0^{N-2}x_1\mu_2+2Nx_0^{N-2}x_2\mu_1.
\end{align*}
Therefore the monomials cannot be divisible by $x_0^{N-1}\mu_1$, $x_0^{N-1}\mu_2$ and $x_0^{N-2}x_1\mu_2$,
and the Poincar\'e series equals
$$
\CP_3^{a=1}(q,t)=\frac{(1-q^{2N-2})(q^{2N+4}t^3+q^{2N+6}t^5)-q^{4N+6}t^7(1-q^4)}{(1-q^2)(1-q^4t^2)(1-q^6t^4)}.
$$
Finally, the homology at level 2 is generated by $\mu_1\mu_2$ and the unique relation has the form:
$$d_N(2\xi_0\xi_1\xi_2)=x_0^{N-2}\mu_1\mu_2,$$
hence
$$
\CP_3^{a=2}(q,t)=\frac{q^{4N+10}t^8(1-q^{2N-4})}{(1-q^2)(1-q^4t^2)(1-q^6t^4)}.
$$
We obtain the following result:
\begin{theorem}
\label{n=3}
The Poincar\'e series for the rational homology $H_{3}^{\ALG}(\SL_N,\BQ)$ is given by the formula
\begin{multline*}
\CP_3(\SL_N,\BQ)(1-q^2)(1-q^4t^2)(1-q^6t^4) = 1-q^{2N}-q^{2N+2}t^2+q^{2N+4}(t^2+t^3-t^4)\\
                                              +q^{2N+6}(t^4+t^5) -q^{4N+2}t^3-q^{4N+4}t^5-q^{4N+6}t^7+q^{4N+10}(t^7+t^8)-q^{6N+6}t^{8}.
\end{multline*}
\end{theorem}

\section{\texorpdfstring{Projectors and $\SL_N$\hy homology of $(3,m)$\hy torus knots}{Projectors and slN homology of (3,m)-torus knots}}
\label{sec3}

In this section we give a conjectural formula for $\SL_N$\hy homology of $(3,m)$\hy torus knots for all $N$ and $m$. 
The construction is motivated by  the study of higher categorified Jones-Wenzl projectors, introduced by Cooper and Hogancamp \cite{CH}, although we do not pursue this analogy here.
These projectors are labeled by the standard Young tableaux (SYT) and categorify the projectors onto irreducible summands in the 
regular representation of the Hecke algebra. By a theorem of Rozansky \cite{Rozproj} (see also \cite{C,rose}), the homology of the categorified projector corresponding to the symmetric representation of $H_n$
coincides with the stable homology of the $(n,\infty)$\hy torus knot.  

\subsection{Decomposition of the homology}

We expect the Poincar\'e polynomials of homology of three-strand torus knots to be defined by the following conjecture. 

\conjTwo*
\begin{table}[p]
\rotatebox{90}{\parbox{\textheight}{
\newcommand{\RH}[1]{\makebox[0.86em]{#1}}
\extrarowsep=1.15mm
\begin{tabu}{|[lightgray]c|*{21}{c|[lightgray]}c||c}
\tabucline[lightgray]{1-23}
\backslashbox{$\ddagger$\kern-0.1em}{\kern-0.1em$t$} & \RH{0} & \RH{1} & \RH{2} & \RH{3} & \RH{4} & \RH{5} & \RH{6} & \RH{7} & \RH{8} & \RH{9}  & \RH{10} & \RH{11} & \RH{12} & \RH{13} & \RH{14} & \RH{15} & \RH{16} & \RH{17} & \RH{18} & \RH{19} & \RH{20} & \RH{21}\\\tabucline[black]{1-23}
6     &   &   &   & 1 &   & 1 &   &   & 1 &   & 1 &   & 1 &   & 1 &   & 1 &   & 1 &   & 1 &    \\\tabucline[lightgray]{1-23}
4     & 1 &   &   & 1 &   & 2 &   & 2 &   & 2 &   & 1 & 1 & 1 & 1 & 1 & 1 & 1 & 1 & 1 & 1 & 1  \\\tabucline[lightgray]{1-23}
$2 $  & 1 &   & 1 &   & 1 & 1 &   & 2 &   & 3 &   & 3 &   & 3 &   & 2 & 1 & 2 & 1 & 2 & 1 & 2  \\\tabucline[lightgray]{1-23}
$0 $  & 1 &   & 1 &   & 2 &   & 2 &   & 2 & 1 & 1 & 2 & 1 & 3 & 1 & 3 & 1 & 3 & 1 & 2 & 2 & 2  \\\tabucline[lightgray]{1-23}
$-2 $ &   &   &   &   & 1 &   & 1 &   & 2 &   & 2 &   & 2 & 1 & 1 & 2 & 1 & 3 & 1 & 3 & 1 & 3 & \quad {\Huge$\cdots$} \\\tabucline[lightgray]{1-23}
$-4 $ &   &   &   &   &   &   &   &   & 1 &   & 1 &   & 2 &   & 2 &   & 2 & 1 & 1 & 2 & 1 & 3  \\\tabucline[lightgray]{1-23}
$-6$  &   &   &   &   &   &   &   &   &   &   &   &   & 1 &   & 1 &   & 2 &   & 2 &   & 2 & 1  \\\tabucline[lightgray]{1-23}
$-8$  &   &   &   &   &   &   &   &   &   &   &   &   &   &   &   &   & 1 &   & 1 &   & 2 &    \\\tabucline[lightgray]{1-23}
$-10$ &   &   &   &   &   &   &   &   &   &   &   &   &   &   &   &   &   &   &   &   & 1 &    \\\tabucline[lightgray]{1-23}
\end{tabu}\\[1ex]
\begin{minipage}{0.5\textheight}
\caption{The unreduced rational $\SL_3$-homology of $T(3,16)$. See the caption of \thmref{table2} for detailed explanations. The first half of the table depicts the stable part, and the second half
the part diverging from $T(3,\infty)$.}
\end{minipage}\hfill
\begin{tabu}{c||*{10}{c|[lightgray]}c|c|[lightgray]}
\tabucline[lightgray]{2-13}
& \RH{22} & \RH{23} & \RH{24}  & \RH{25} & \RH{26} & \RH{27} & \RH{28} & \RH{29} & \RH{30} & \RH{31} & \RH{32} & \slashbox{\kern-0.1em$t$}{$\ddagger$\kern-0.1em}\\\tabucline[black]{2-13}
                    &   &   &   &   &   &   &   &   &   &   &    & 6      \\\tabucline[lightgray]{2-13}
                    & 1 &   & 1 &   & 1 &   & 1 &   & 1 &   & 1  & 4      \\\tabucline[lightgray]{2-13}
                    & 1 & 1 & 1 & 1 & 1 & 1 & 1 & 1 & 1 & 1 &    & 2      \\\tabucline[lightgray]{2-13}
                    & 1 & 2 & 1 & 2 & 1 & 2 & 1 & 2 &   & 1 &    & 0      \\\tabucline[lightgray]{2-13}
{\Huge\ldots}\qquad & 2 & 2 & 2 & 2 & 2 & 2 & 1 & 1 & 1 &   &    & $-2$  \\\tabucline[lightgray]{2-13}
                    & 1 & 2 & 2 & 2 & 1 & 1 & 1 &   &   &   &    & $-4 $  \\\tabucline[lightgray]{2-13}
                    & 1 & 2 & 1 & 1 & 1 &   &   &   &   &   &    & $-6 $  \\\tabucline[lightgray]{2-13}
                    & 1 & 1 & 1 &   &   &   &   &   &   &   &    & $-8 $  \\\tabucline[lightgray]{2-13}
                    & 1 &   &   &   &   &   &   &   &   &   &    & $-10$  \\\tabucline[lightgray]{2-13}
\end{tabu}
}}
\end{table}

The conjecture is expected to hold both for reduced and unreduced 
$\SL_N$\hy and HOMFLY-PT homology of $(3,m)$\hy torus knots, and we prove below that a similar decomposition holds for the Heegaard-Floer homology, too. 
For $N=2$, the answer agrees with the computations of Turner \cite{turner}. For $N=3$, we checked the answer on the computer both for large ($k=28$) and small 
torus knots and found a perfect agreement with the experimental data.

The existence of such decompositions was first conjectured by Oblomkov and Rasmussen \cite{OR}, but, to the knowledge of the authors, the explicit formulae for $\CP_T$ have not been written down before.

\subsection{Symmetric projector, two strands}

The unreduced HOMFLY-PT homology of the symmetric projector is a free algebra with generators $x_0,x_1,\xi_0,\xi_1$ with the degrees:
$$
\deg(x_0)=q^2,\quad \deg(x_1)=q^4t^2,\quad \deg(\xi_0)=a^2t,\quad \deg(\xi_1)=a^2q^2t^3.
$$
Therefore the Poincar\'e series of this homology equals:
$$
\CP_{\text{\scriptsize\young(12)}}=\frac{(1+a^2t)(1+a^2q^2t^3)}{(1-q^2)(1-q^4t^2)}.
$$
The reduced HOMFLY-PT homology is generated by $x_1$ and $\xi_1$, and the Poincar\'e series equals:
$$
\CP_{\text{\scriptsize\young(12)}}^{\RED}=\frac{(1+a^2q^2t^3)}{(1-q^4t^2)}.
$$
According to \thmref{main conj}, the $\SL_N$\hy homology can be described as the homology of the differential $d_N$ after regrading $a=q^N$.
The unreduced $\SL_N$\hy homology of the symmetric projector are then given by \eqref{n=2}:
$$
\CP_{\text{\scriptsize\young(12)}}^{d_N}=\frac{1-q^{2N}-q^{2N+2}t^2+q^{2N+4}t^2+q^{2N+4}t^3-q^{4N+2}t^3}{(1-q^2)(1-q^4t^2)}.
$$
Remark that the Euler characteristic of this projector is equal to the $S^2$-colored $\SL_N$\hy invariant of the unknot.
In the reduced homology one has $d_N=0$, so 
 $$
\CP_{\text{\scriptsize\young(12)}}^{d_N,\RED}=\frac{(1+q^{2N+2}t^3)}{(1-q^4t^2)}.
$$

\subsection{Antisymmetric projector, two strands}

The homology of the antisymmetric projector can be computed as follows.
The HOMFLY-PT homology is an algebra with two even generators $x_0,a_1$ and two odd generators $\xi_0,\theta_1$.
The degrees are equal to:
$$
\deg x_0=q^2,\quad \deg a_1=q^{-4}t^{-2},\quad\deg \xi_0=a^2t,\quad \deg \theta_1=q^{-2}a^2t. 
$$
Therefore the Poincar\'e series of this homology equals:
$$
\CP_{\text{\scriptsize\young(1,2)}}=\frac{(1+a^2t)(1+a^2q^{-2}t)}{(1-q^2)(1-q^{-4}t^{-2})}.
$$
The reduced Poincar\'e series equals
$$
\CP_{\text{\scriptsize\young(1,2)}}^{\RED}=\frac{(1+a^2q^{-2}t)}{(1-q^{-4}t^{-2})}.
$$ 
The differential is given by the formula 
$$d_N(\xi_0)=x_0^{N},\quad d_N(\theta_1)=Nx_0^{N-1}.$$
The homology is generated by $x_0,a_1$ and $N\xi_0-x_0\theta_1$ modulo relation $x_0^{N-1}=0$.
The Poincar\'e series has the form:
$$
\CP_{\text{\scriptsize\young(1,2)}}^{d_N}=\frac{(1-q^{2N-2})(1+q^{2N}t)}{(1-q^2)(1-q^{-4}t^{-2})}.
$$
The reduced Poincar\'e series equals:
$$
\CP_{\text{\scriptsize\young(1,2)}}^{d_N,\RED}=\frac{(1+q^{2N-2}t)}{(1-q^{-4}t^{-2})}.
$$ 
There is a natural embedding $\CH^{\ALG}(1,\infty)\hookrightarrow \CH^{\ALG}(2,\infty)$, and the Poincar\'e series of the quotient 
equals $-\CP_{\text{\scriptsize\young(1,2)}}.$ In other words, we have an identity
$$
\CP_{\text{\scriptsize\young(12)}}+\CP_{\text{\scriptsize\young(1,2)}}=\CP_{\text{\scriptsize\young(1)}}.
$$
Indeed, in $\CH^{\ALG}(1,\infty)$ we have generators $\xi_0,x_0$ and the differential $d_N(\xi_0)=x_0^{N}$.
In $\CH^{\ALG}(2,\infty)$ we add the generators $\xi_1,x_1$ and the differential $d_N(\xi_1)=Nx_0^{N-1}x_1$.
The quotient is spanned by the monomials containing either $x_1$ or $\xi_1$. If we introduce the formal variable $\theta_1=\xi_1/x_1$,
then the quotient would be spanned by all monomials divisible by $x_1$. On the other hand,
$d_N(\theta_1)=Nx_0^{N-1}$.

\subsection{Homology of $(2,m)$\hy torus knots}

The  unreduced $\SL_N$\hy homology of $(2,m)$\hy torus knots was computed by Cautis in \cite[Example 10.3.3]{C}, and can be reformulated as 
following:

\begin{theorem}(\cite{C})
The Poincar\'e polynomial of the $\SL_N$\hy homology of the $(2,2k+1)$\hy torus knot is given by the following equation:
$$
\CP(T(2,2k+1), \SL_N)=\CP_{\text{\scriptsize\young(12)}}^{d_N}+q^{4k}t^{2k}\CP_{\text{\scriptsize\young(1,2)}}^{d_N}.
$$
The same decomposition holds for the reduced homology.
\end{theorem}

The reduced homology of $(2,2k+1)$\hy torus knots is well known, see e.g. \cite{DGR}.

\subsection{Symmetric projector, three strands}
As above, the HOMFLY-PT homology of the symmetric projector is expected to coincide with the algebra $\CH_3$, its Poincar\'e series equals (cf. \cite{DGR}):
$$
\CP_{\text{\scriptsize\young(123)}}=\frac{(1+a^2t)(1+a^2q^2t^3)(1+a^2q^4t^5)}{(1-q^2)(1-q^4t^2)(1-q^6t^4)},\quad
\CP_{\text{\scriptsize\young(123)}}^{\RED}=\frac{(1+a^2q^2t^3)(1+a^2q^4t^5)}{(1-q^4t^2)(1-q^6t^4)}.
$$
The unreduced $\SL_N$\hy homology of the symmetric projector is given by \thmref{n=3}:
\begin{align*}
\CP_{\text{\scriptsize\young(123)}}^{d_N}=\ &\frac{1}{(1-q^2)(1-q^4t^2)(1-q^6t^4)}(1-q^{2N}-q^{2N+2}t^2+q^{2N+4}(t^2+t^3-t^4)+\\[2ex]
 &        q^{2N+6}(t^4+t^5)  -q^{4N+2}t^3-q^{4N+4}t^5-q^{4N+6}t^7+q^{4N+10}(t^7+t^8)-q^{6N+6}t^{8}).
\end{align*}
In the reduced homology we have $d_N=0$ for $N>2$ and $d_2(\xi_2)=x_1^2$. Therefore
$$
\CP_{\text{\scriptsize\young(123)}}^{d_N,\RED}=\frac{(1+q^{2N+2}t^3)(1+q^{2N+4}t^5)}{(1-q^4t^2)(1-q^6t^4)},\ N>2;\quad
\CP_{\text{\scriptsize\young(123)}}^{d_2,\RED}=\frac{(1-q^8t^4)(1+q^{6}t^3)}{(1-q^4t^2)(1-q^6t^4)}.
$$
We also introduce a differential $d_0$ on the reduced homology by the formula $d_0(\xi_2)=x_1$.
It is clear that
$$
\CP_{\text{\scriptsize\young(123)}}^{d_0,\RED}=\frac{(1+a^2q^2t^3)}{(1-q^6t^4)}.
$$

\subsection{Antisymmetric projector, three strands}

The HOMFLY-PT homology of the antisymmetric projector is an algebra with two even generators $x_0,a_1,a_2$ and two odd generators $\xi_0,\theta_1,\theta_2$.
The degrees are equal to:
$$
\deg x_0=q^2,\quad \deg a_1=q^{-4}t^{-2},\quad \deg a_2=q^{-6}t^{-2},$$
$$\deg \xi_0=a^2t,\quad \deg \theta_1=a^2q^{-2}t,\quad \deg \theta_2=a^2q^{-4}t. $$
The Poincar\'e series equals:
$$
\CP_{\text{\scriptsize\young(1,2,3)}}=\frac{(1+a^2t)(1+a^2q^{-2}t)(1+a^2q^{-4}t)}{(1-q^2)(1-q^{-4}t^{-2})(1-q^{-6}t^{-2})},\quad
\CP_{\text{\scriptsize\young(1,2,3)}}^{\RED}=\frac{(1+a^2q^{-2}t)(1+a^2q^{-4}t)}{(1-q^{-4}t^{-2})(1-q^{-6}t^{-2})}.
$$
The differential is given by the formula 
$$
d_N(\xi_0)=x_0^{N},\quad d_N(\theta_1)=Nx_0^{N-1},\quad d_N(\theta_2)=\binom{N}{2}x_0^{N-2}.
$$
The homology has even generators $x_0,a_1,a_2$ and two odd generators $N\xi_0-x_0\theta_1,\frac{N-1}{2}\theta_1-x_0\theta_2$ modulo relation $x_0^{N-2}=0$, so
$$
\CP_{\text{\scriptsize\young(1,2,3)}}^{d_N}=\frac{(1-q^{2N-4})(1+q^{2N-2}t)(1+q^{2N}t)}{(1-q^2)(1-q^{-4}t^{-2})(1-q^{-6}t^{-2})}.
$$
Note that for $N=2$ this homology vanishes, as expected. In reduced homology $d_N=0$, so
$$
\CP_{\text{\scriptsize\young(1,2,3)}}^{d_N,\RED}=\frac{(1+q^{2N-2}t)(1+q^{2N}t)}{(1-q^{-4}t^{-2})(1-q^{-6}t^{-2})}.
$$
The differential $d_0$ on reduced homology is given by the formula $d_0(\theta_2)=a_1,$ and
$$
\CP_{\text{\scriptsize\young(1,2,3)}}^{d_0,\RED}=\frac{(1+a^2q^{-2}t)}{(1-q^{-6}t^{-2})}.
$$

\subsection{Hook-shaped projectors}

We  formally  introduce two power series $\CP_{\text{\scriptsize\young(12,3)}}$ and $\CP_{\text{\scriptsize \young(13,2)}}$
such that the following equations hold:
$$
\CP_{\text{\scriptsize\young(12)}}=\CP_{\text{\scriptsize\young(123)}}+\CP_{\text{\scriptsize\young(12,3)}},\quad \CP_{\text{\scriptsize\young(1,2)}}=\CP_{\text{\scriptsize\young(1,2,3)}}+\CP_{\text{\scriptsize\young(13,2)}}.
$$
One can check that for HOMFLY-PT homology these series are given by
$$
\CP_{\text{\scriptsize\young(12,3)}}=\frac{(1+a^2t)(1+a^2q^{-2}t)(1+a^2q^2t^3)}{(1-q^2)(1 - q^4 t^2) (1 - q^{-6} t^{-4})},
$$
$$
\CP_{\text{\scriptsize\young(13,2)}}=\frac{(1+a^2t)(1+a^2q^{-2}t)(1+a^2q^2t^3)}{(1-q^2)(1 - q^{-4}t^{-2}) (1 - q^{6} t^{2})},
$$
$$
\CP_{\text{\scriptsize\young(12,3)}}^{\RED}=\frac{ (1+a^2q^{-2}t)(1+a^2q^2t^3)}{ (1 - q^4 t^2) (1 - q^{-6} t^{-4})},\quad 
\CP_{\text{\scriptsize\young(13,2)}}^{\RED}=\frac{ (1+a^2q^{-2}t)(1+a^2q^2t^3)}{ (1 - q^{-4}t^{-2}) (1 - q^{6} t^{2})}.
$$
For the unreduced $\SL_N$\hy homology we have:
\begin{equation}
\label{hook1}
\CP_{\text{\scriptsize\young(12,3)}}^{d_N}=\frac{(1 + q^{2N}t) (1 - q^{2N-2} - q^{2N+2} t^2 + q^{2N + 4}t^2 +
     q^{2N + 4} t^3 - 
    q^{4 N} t^3)}{(1 - q^2) (1 - q^4 t^2) (1 - q^{-6} t^{-4})}
\end{equation}
\begin{equation}
\label{hook2}
\CP_{\text{\scriptsize \young(13,2)}}^{d_N}=\frac{(1 + q^{2N}t) (1 - q^{2N-2} - q^{2N+2} t^2 + q^{2N + 4}t^2 +
     q^{2N + 4} t^3 - 
    q^{4 N} t^3)}{(1 - q^2) (1 - q^{-4}t^{-2}) (1 - q^{6} t^{2})}.
\end{equation}
For the reduced $\SL_N$\hy homology we have
\[
\begin{alignedat}{4}
& \CP_{\text{\scriptsize\young(12,3)}}^{d_N,\RED}=\frac{ (1+q^{2N-2}t)(1+q^{2N+2}t^3)}{ (1 - q^4 t^2) (1 - q^{-6} t^{-4})},\quad
&& \CP_{\text{\scriptsize\young(13,2)}}^{d_N,\RED}=\frac{ (1+q^{2N-2}t)(1+q^{2N+2}t^3)}{ (1 - q^{-4}t^{-2}) (1 - q^{6} t^{2})}\ (N>2); \\
& \CP_{\text{\scriptsize\young(12,3)}}^{d_2,\RED}=\frac{ (1-q^2)(1+q^{6}t^3)}{ (1 - q^4 t^2) (1 - q^{-6} t^{-4})},
&& \CP_{\text{\scriptsize\young(13,2)}}^{d_2,\RED}=\frac{ (1+q^{2}t)}{ (1 - q^{-4}t^{-2})}=\CP_{\text{\scriptsize\young(1,1)}}^{d_2,\RED}.
\end{alignedat}
\]
These can also be interpreted as homology of $d_N$ acting on some algebras.

For the diagram $\text{\scriptsize\young(12,3)}$, the HOMFLY-PT homology
is generated by $x_0,x_1,b_2,\xi_0,\xi_1,\theta_1$, and the differential has the form
$$
d_N(\xi_0)=x_0^N,\quad d_N(\xi_1)=Nx_0^{N-1}x_1,\quad d_N(\theta_1)=Nx_0^{N-1}+\binom{N}{2}x_0^{N-2}x_1^2b_2.
$$
Note that this differential can be obtained from the symmetric one: we set $\theta_1=\xi_2/x_2$ and $b_2=1/x_2$, so
$$
d_N(\theta_1)=\frac{1}{x_2}d_N(\xi_2)=\frac{1}{x_2}(Nx_0^{N-1}x_2+\binom{N}{2}x_0^{N-2}x_1^2).
$$
For the diagram $\text{\scriptsize \young(13,2)}$, the HOMFLY-PT homology
is generated by $x_0,a_1,x_2,\xi_0,\theta_1,\xi_2$ and the differential has the form
$$
d_N(\xi_0)=x_0^{N-1},\quad d_N(\theta_1)=Nx_0^{N-1},\quad d_N(\xi_2)=\binom{N}{2}x_0^{N-2}x_2.
$$
One can check that the Poincar\'e series for the homology of $d_N$ agree with \eqref{hook1} and \eqref{hook2}.

Finally, we define $d_0$ for the diagram $\text{\scriptsize\young(12,3)}$ by the equation $d_0(\theta_1)=x_1b_2$,
and for the diagram $\text{\scriptsize\young(12,3)}$ by the equation $d_0(\xi_2)=x_2a_1$.
One can check that
$$
\CP_{\text{\scriptsize\young(12,3)}}^{d_0,\RED}=\frac{(1 - q^{-2}t^{-2}) (1 + a^2q^2t^3)}{(1 - q^4 t^2) (1 - q^{-6} t^{-4})},\quad 
\CP_{\text{\scriptsize\young(13,2)}}^{d_0,\RED}=\frac{(1-q^2)(1+a^2q^{-2}t)}{(1 - q^{-4} t^{-2})(1 - q^6 t^2)}. 
$$

\subsection{Homology of $(3,m)$\hy torus knots}

We summarize the known evidence for the \thmref{conj:projectors} in the following theorem.

\begin{theorem}
\begin{enumerate}
\item[1.] For the HOMFLY-PT homology, \thmref{conj:projectors} agrees with the ``refined Chern-Simons invariants'' defined in \cite{AS}.
The conjecture is equivalent to the main conjecture of \cite{GORS}.

\item[2.] The conjecture holds for $\SL_2$\hy homology.

\item[3.] The reduced homology of $d_0$ agrees with the hat version of the Heegaard-Floer homology of $(3,m)$\hy torus links
after regrading $a=t^{-1}$ (cf. \cite[Section 3.8]{DGR}).
\end{enumerate}
\end{theorem}

\begin{proof}
The refined Chern-Simons invariants have been been presented as sums over SYT in \cite{GN}, and the conjectural HOMFLY-PT homology 
of $(3,m)$\hy torus knots was first described in \cite[Conjecture 6.8]{DGR} and later reformulated in \cite{GORS}. The $\SL_2$\hy homology were computed by Turner in \cite{turner},
and one can check his answers agree with the ones provided by  \thmref{conj:projectors}:
\begin{multline*}
\CP(T(3,3k+1),d_2)=\frac{1+q^2+q^4t^2+q^8t^3+q^{10}t^5+q^{12}t^5}{1-q^6t^4}+\\ q^{6k}t^{4k}\frac{(1-q^6t^2)(1+q^4t)}{(1-q^4t^2)(1-q^{-6}t^{-4})}+q^{6k}t^{4k}\frac{1+q^4t}{1-q^{-4}t^{-2}},
\end{multline*}
\begin{multline*}
\CP(T(3,3k+2),d_2)=\frac{1+q^2+q^4t^2+q^8t^3+q^{10}t^5+q^{12}t^5}{1-q^6t^4}+\\ q^{6k}t^{4k}\frac{(1-q^6t^2)(1+q^4t)}{(1-q^4t^2)(1-q^{-6}t^{-4})}+q^{6k+4}t^{4k+2}\frac{1+q^4t}{1-q^{-4}t^{-2}}.
\end{multline*}
Similar decomposition for $\SL_2$ homology of $(3,m)$\hy torus knots was obtained in \cite{OR}.
The Heegaard-Floer homology of torus knots is well known, see e.g \cite[Section 6.12]{DGR} for its description for $(3,m)$\hy torus knots.
One can compare it with:
\begin{multline*}
\CP^{\RED}(T(3,3k+1),d_0)=\frac{1+q^{2}t}{1-q^6t^4}+q^{6k}t^{4k}\frac{(1-q^{-2}t^{-2})(1+q^2t)}{(1-q^4t^2)(1-q^{-6}t^{-4})}+\\
q^{6k}t^{4k}\frac{(1-q^2)(1+q^{-2}t^{-1})}{(1-q^{-4}t^{-2})(1-q^6t^2)}+q^{12k}t^{6k}\frac{1+q^{-2}t^{-1}}{1-q^{-6}t^{-2}}, 
\end{multline*}
\begin{multline*}
\CP^{\RED}(T(3,3k+2),d_0)=\frac{1+q^{2}t}{1-q^6t^4}+q^{6k}t^{4k}\frac{(1-q^{-2}t^{-2})(1+q^2t)}{(1-q^4t^2)(1-q^{-6}t^{-4})}+\\
q^{6k+4}t^{4k+2}\frac{(1-q^2)(1+q^{-2}t^{-1})}{(1-q^{-4}t^{-2})(1-q^6t^2)}+q^{12k+4}t^{6k+2}\frac{1+q^{-2}t^{-1}}{1-q^{-6}t^{-2}}.
\end{multline*}
\end{proof}

For $N=3$, we checked the conjecture up to the $(3,83)$\hy torus knot.

\section{\texorpdfstring{Data for stable $\SL_3$\hy homology}{Data for stable sl3-homology}}
\label{sec4}
In this section, the stable $\SL_3$\hy homology of torus knots is compared against FoamHo calculations \cite{foam}, which
were conducted on a Xeon CPU E5-2620 with 128 GB RAM. 
Note that the quantum degree of the stable part of the $\SL_3$\hy homology of an actual $(n,m)$-torus knot is shifted
by $2(nm - n - m)$ relative to the homology of the $(n,\infty)$-torus knot.

\subsection{$\mathbf{n=2}$}

The unreduced Poincar\'e series is given by \eqref{n=2}:
$$
\CP_2^{\ALG}(\SL_3,\BQ)=\frac{1 - q^6 - q^8 t^2 + q^{10}t^2 + q^{10}t^3 - q^{14}t^3}{(1 - q^2) (1 - q^4t^2)}.
$$
The reduced Poincar\'e series equals:
$$
\CP_2^{\ALG,\RED}(\SL_3,\BQ)=\frac{1 + q^8 t^3}{1 - q^4 t^2}.
$$
All of homology is stable, e.g. the first divergence when comparing with the $(2,201)$\hy torus knot
occurs at homological degree $202$.

\subsection{$\mathbf{n=3}$}

The unreduced Poincar\'e series is given by \thmref{n=3}:

\begin{multline*}
\CP_3^{\ALG}(\SL_3,\BQ)(1 - q^2) (1 - q^4 t^2) (1 - q^6 t^4)= \\
1 - q^6 - q^8 t^2 + q^{10}t^2 + q^{10}t^3 - q^{14}t^3 - q^{10}t^4 + 
   q^{12}t^4 + q^{12}t^5 - q^{16}t^5 - q^{18}t^7 + q^{22}t^7 + q^{22}t^8 - 
   q^{24}t^8.
\end{multline*}
The reduced Poincar\'e series equals:
$$
\CP_3^{\ALG,\RED}(\SL_3,\BQ)=\frac{(1 + q^8 t^3) (1 + q^{10}t^5)}{(1 - q^4 t^2) (1 - q^6 t^4)}.
$$
Compared with the rational homology of the $(3,83)$\hy torus knot (calculation took 6 hours and 1.5 GB RAM),
the first divergence is at $q^{168}t^{112}$, both for unreduced and for reduced homology.

\subsection{$\mathbf{n=4}$}

The computations of the Koszul homology were done with Singular \cite{singular},
a computer algebra system. The unreduced Poincar\'e series equals:

\begin{multline*}
\CP_4^{\ALG}(\SL_3,\BQ)(1 - q^2) (1 - q^4 t^2) (1 - q^6 t^4) (1 - q^8 t^6)= \\
1 - q^6 - q^8 t^2 + q^{1} t^2 + q^{10}t^3 - q^{14}t^3 - q^{10}t^4 + 
   q^{12}t^4 + q^{12}t^5 - q^{16}t^5 - q^{12}t^6 + q^{14}t^6 + q^{14}t^7 - 
   2 q^{18}t^7 + \\
       q^{22}t^7 + q^{22}t^8 - q^{24}t^8 - q^{20}t^9 + q^{24}t^9 +
    q^{24}t^{10} - q^{26}t^{10}- q^{22}t^{11}+ q^{26}t^{11} + q^{26}t^{12} - 
   q^{28}t^{12} - q^{30}t^{14} + 
   q^{36}t^{14}
\end{multline*}
The reduced Poincar\'e series equals:
$$
\CP_4^{\ALG,\RED}(\SL_3,\BQ)=\frac{(1 + q^8 t^3) (1 + q^{10}t^5) (1 - q^{12}t^6)}{(1 - q^4 t^2) (1 - 
   q^6 t^4) (1 - q^8 t^6)}.
$$
Compared with the rational homology of the $(4,35)$\hy torus knot (computation took 12 days and 25 GB RAM),
the first divergence is at $q^{72}t^{54}$, both for unreduced and for reduced homology.

\subsection{$\mathbf{n=5}$}

The unreduced Poincar\'e series (again computed by Singular) equals:

\begin{multline*}
\CP_5^{\ALG}(\SL_3,\BQ)(1 - q^2) (1 - q^4 t^2) (1 - q^6 t^4) (1 - 
  q^8 t^6) (1 - q^{10} t^8)= \\
     1 - q^6 - q^8 t^2 + q^{10} t^2 + q^{10} t^3 - q^{14} t^3 - q^{10} t^4 + 
q^{12} t^4 + q^{12} t^5 - q^{16} t^5 - q^{12} t^6 + q^{14} t^6 + q^{14} t^7 - 
2 q^{18} t^7 + \\
    +q^{22} t^7 - q^{14} t^8 + q^{16} t^8 + q^{22} t^8 - q^{24} t^8 +
 q^{16} t^9 - 2 q^{20} t^9 + q^{24} t^9 + q^{24} t^{10} - q^{26} t^{10} - 
2 q^{22} t^{11} + 2 q^{26} t^{11} +\\
    2 q^{26} t^{12} - 2 q^{28} t^{12} - q^{24} t^{13} +
 q^{28} t^{13} + q^{26} t^{14} - 2 q^{30} t^{14} + q^{36} t^{14} - q^{28}t^{15} + 
q^{30}t^{15} + q^{32}t^{15} - q^{34} t^{15} + q^{30} t^{16} \\
    -q^{32} t^{16} - 
q^{34}t^{16} + q^{38}t^{16} + q^{34} t^{17} - q^{36} t^{17} - q^{36} t^{18} + 
2 q^{40}t^{18} - q^{42} t^{18} + q^{36} t^{19} - q^{38} t^{19} + q^{40} t^{19} - 
q^{42} t^{19} \\
    - q^{38} t^{20} + 2 q^{42} t^{20} - q^{44} t^{20} + q^{42} t^{21} - 
q^{44}t^{21}+ q^{44}t^{22} - q^{46} t^{22} - q^{46} t^{23} + 
q^{50}t^{23}.
\end{multline*}
The reduced Poincar\'e series equals:
$$
\CP_5^{\ALG,\RED}(\SL_3,\BQ)=\frac{(1 + q^8 t^3) (1 + q^{10} t^5) (1 - q^{12} t^6 - q^{14} t^8 + q^{18} t^{10} + 
q^{18} t^{11} - q^{26} t^{15})}{(1 - q^4 t^2) (1 - q^6 t^4) (1 - 
    q^8 t^6) (1 - q^{10} t^8)}
$$
Compared with the rational homology of the $(5,14)$\hy torus knot (computation took 15 days and 33 GB RAM),
the first divergence is at $q^{30} t^{24}$, both for unreduced and for reduced homology.

\newcommand{\arXiv}[1]{\href{http://arxiv.org/abs/#1}{arXiv:#1}}


\begin{thebibliography}{99}

\bibitem{AS} M. Aganagic, S. Shakirov. \emph{Refined Chern-Simons theory and knot homology}.  String-Math 2011, 3--31, 
Proc. Sympos. Pure Math., {\bf 85}, Amer. Math. Soc., Providence, RI, 2012.

\bibitem{BN} D. Bar-Natan. \emph{Fast Khovanov Homology Computations}. J. Knot Theory Ramifications {\bf 16}
(2007), no. 3, 243--255.

\bibitem{katlas} D. Bar-Natan, S. Morrison. \emph{The Knot Atlas}. \url{http://katlas.org}.


\bibitem{C} S. Cautis. \emph{Clasp technology to knot homology via the affine Grassmannian}. \arXiv{1207.2074}.

\bibitem{CH} B. Cooper, M. Hogancamp. \emph{An Exceptional Collection For Khovanov Homology}.  \arXiv{1209.1002}.

\bibitem{singular}
W. Decker, G.-M. Greuel, G. Pfister, H. Sch{\"o}nemann.
\newblock \emph{{\sc Singular} {3-1-3} --- {A} computer algebra system for polynomial computations}.
\newblock \url{http://www.singular.uni-kl.de} (2011).

\bibitem{DGR} N. Dunfield, S. Gukov, J. Rasmussen. \emph{The superpolynomial for knot homologies}. 
Experiment. Math. {\bf 15} (2006), no. 2, 129--159. 

\bibitem{GGS} E. Gorsky, S. Gukov, M. Sto\v si\'c. \emph{Quadruply-graded colored homology of knots}. \arXiv{1304.3481}.

\bibitem{GN} E. Gorsky, A. Negut. \emph{Refined knot invariants and Hilbert schemes}.  \arXiv{1304.3328}. 

\bibitem{GOR} E. Gorsky, A. Oblomkov, V. Shende. \emph{On stable Khovanov homology of torus knots}. Experiment. Math. {\bf 22} (2013), no. 3, 265--281.

\bibitem{GORS} E. Gorsky, A. Oblomkov, J. Rasmussen, V. Shende. \emph{Torus knots and the rational DAHA}. \arXiv{1207.4523}.

\bibitem{RJ} V. Jones, M. Rosso. \emph{On the invariants of torus knots derived from quantum groups}. 
J. Knot Theory Ramifications {\bf 2} (1993), no. 1, 97--112. 


\bibitem{L} L. Lewark. \emph{$\mathfrak{sl}_3$-foam homology calculations}. Algebr. Geom. Topol. {\bf 13} (2013), no. 6, 3661--3686.

\bibitem{foam} L. Lewark. \emph{{\sc FoamHo}, an $\mathfrak{sl}_3$-homology calculator}.  \url{http://www.maths.dur.ac.uk/~vxhn54/foamho.html} (2012).

\bibitem{OR} A. Oblomkov, J. Rasmussen, unpublished, 2011. 

\bibitem{PS} J. Przytycki, R. Sazdanovi\'c. \emph{Torsion in Khovanov homology of semi-adequate links}.  \arXiv{1210.5254}.

\bibitem{rose} D. Rose. \emph{A Categorification of Quantum $sl_3$ Projectors and the $sl_3$ Reshetikhin-Turaev Invariant of Tangles}. \arXiv{1109.1745}.

\bibitem{Rozproj} L. Rozansky. \emph{An infinite torus braid yields a categorified Jones-Wenzl projector}. \arXiv{1005.3266}. 

\bibitem{shu} A.  Shumakovitch. \emph{Torsion of the Khovanov homology}. \arXiv{math/0405474}.

\bibitem{turner} P. Turner. \emph{A spectral sequence for Khovanov homology with an application to $(3,q)$--torus links}.
Algebr. Geom. Topol. {\bf 8} (2008) 869--884.

\end{thebibliography}
\end{document}